\documentclass[11pt,a4paper]{article}

\usepackage{amsthm,amsmath,amsfonts,amssymb}
\usepackage[boxed,linesnumbered]{algorithm2e}
\usepackage{color}
\usepackage{float}
\usepackage{graphicx}
\usepackage{url}
\usepackage{latexsym,epsfig,array}

\newcommand{\NN}{\mathbb{N}}

\newcommand{\spb}[1]{\smallskip}
\newcommand{\mpb}[1]{\medskip}
\newcommand{\bpb}[1]{\bigskip}

\DeclareGraphicsExtensions{.jpg,.pdf,.mps,.png}
\title{New results for the Mondrian art problem}
\author{C. Dalf\'o$^a$, M. A. Fiol$^b$, N. L\'opez$^c$ \\
{\small $^a$Dept. de Matem\`atica, Universitat de Lleida}\\
{\small Igualada (Barcelona), Catalonia}\\
{\small {\tt cristina.dalfo@udl.cat}}\\
{\small $^{b}$Dept. de Matem\`atiques, Universitat Polit\`ecnica de Catalunya} \\
{\small Barcelona Graduate School of Mathematics} \\
{\small Barcelona, Catalonia} \\
{\small {\tt miguel.angel.fiol@upc.edu}} \\
{\small $^c$Dept. de Matem\`atica, Universitat de Lleida}\\
{\small Lleida, Spain}\\
{\small {\tt nacho.lopez@udl.cat}}\\
}

\newtheorem{proposition}{Proposition}[section]
\newtheorem{corollary}{Corollary}[section]
\newtheorem{lemma}{Lemma}[section]
\newtheorem{theorem}{Theorem}[section]

\newcommand\blfootnote[1]{%
	\begingroup
	\renewcommand\thefootnote{}\footnote{#1}%
	\addtocounter{footnote}{-1}%
	\endgroup
}

\begin{document}

\date{\today}

\maketitle

\blfootnote{
	\begin{minipage}[l]{0.28\textwidth} \includegraphics[trim=10cm 6cm 10cm 5cm,clip,scale=0.15]{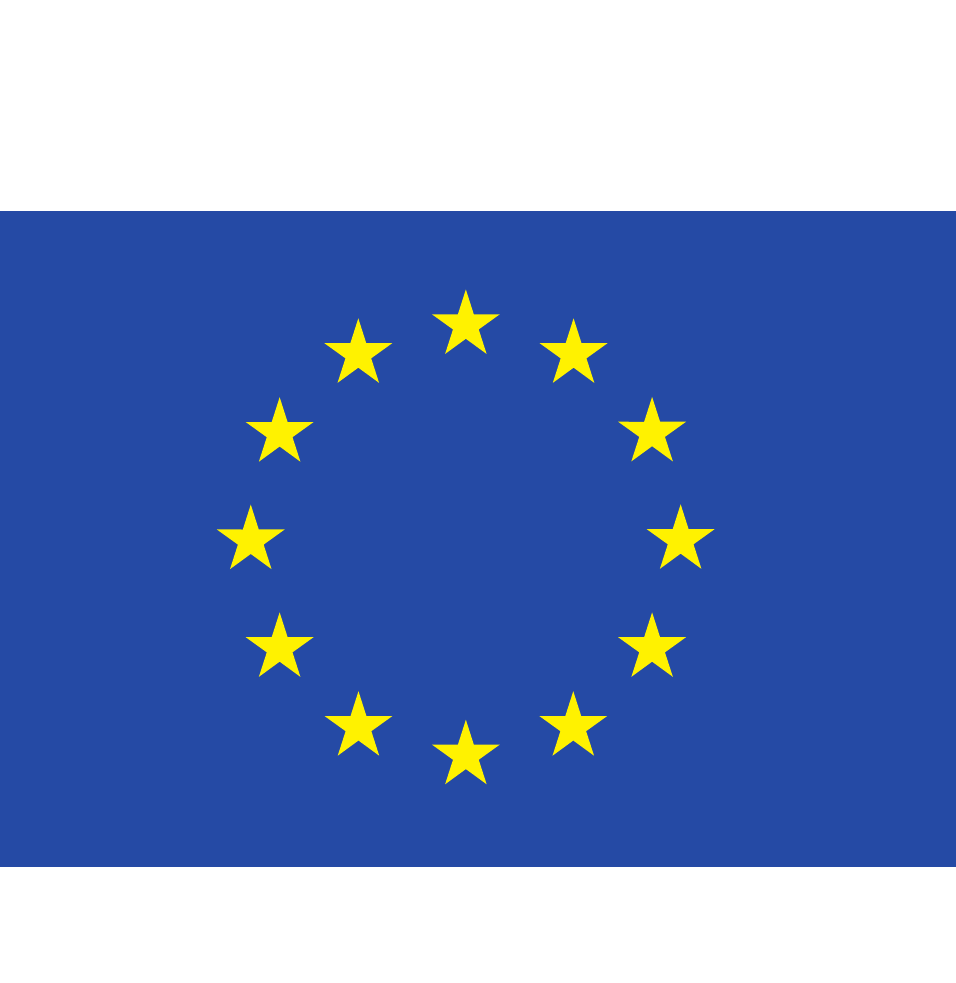} \end{minipage}  \hspace{-2cm} \begin{minipage}[l][1cm]{0.79\textwidth}
		The research of the first author has also received funding from the European Union's Horizon 2020 research and innovation programme under the Marie Sk\l odowska-Curie grant agreement No 734922.
\end{minipage}
}

\begin{abstract}
The Mondrian problem consists of dissecting a square of side length $n\in \NN$ into non-congruent rectangles with natural length sides such that the difference $d(n)$ between the largest and the smallest areas of the rectangles partitioning the square is minimum. In this paper,  we compute some bounds on $d(n)$ in terms of the number of rectangles of the square partition. These bounds provide us optimal partitions for some values of $n \in \NN$. We provide a sequence of square partitions such that $d(n)/n^2$ tends to zero for $n$ large enough. For the case of `perfect' partitions,  that is, with 
$d(n)=0$, we show that, for any fixed powers $s_1,\ldots, s_m$, a square with side length $n=p_1^{s_1}\cdots p_m^{s_m}$, can have a perfect Mondrian partition only if $p_1$ satisfies a given lower bound.   
Moreover,  if $n(x)$ is the number of side lengths $x$ (with $n\le x$) of squares not having a perfect partition, we prove that its `density' $\frac{n(x)}{x}$ is asymptotic to $\frac{(\log(\log(x))^2}{2\log x}$, which improves previous results.
\end{abstract}

\noindent\emph{Keywords:} Non-congruent rectangles, partition, Mondrian problem.

\noindent\emph{MSC 2010:} 05A17.

\section{Introduction}

Let us consider a two-dimensional grid of dimensions $n \times n$, for a given positive integer $n$. We want to partition the whole area of this grid by using only rectangles of different integer dimensions. Overlapping rectangles is not allowed, and once a rectangle of dimensions $a \times b$ has been used, we cannot use another one either with dimensions $a \times b$ or $b \times a$ (a ninety-degree rotation of the original rectangle). When the whole area of the grid has been filled by rectangles, we score it by computing the difference $d(n)$ between the areas of the largest and the smallest rectangles. The Mondrian art problem is to find the smallest $d(n)$ among all possible fillings of the grid. This problem is based on the artwork of the artist Piet Mondrian (1872-1944). In his famous paintings, Mondrian uses canvas tessellated by primary-colored rectangles. Paraphrasing O'Kuhn \cite{okuhn2018mondrian}: `The idea of the puzzle is that an art critic has ordered Mondrian only to create paintings whose rectangles are all incongruent to one another and only have integer side lengths. Furthermore, he can only use a square canvas whose side length is also an integer. Aggravated, Mondrian still wants to create works whose areas of the rectangles are all as close as possible. The work contained in \cite{okuhn2018mondrian} is partial progress towards answering the question of whether $d(n)$ can ever equal 0. Besides, optimal solutions for $n \leq 32$ are given in Bassen \cite{Bassen16} using an algorithmic approach. Similar problems, like the decomposition of a square into rectangles of minimal perimeter \cite{KMW87}, have been studied before.

\subsection*{A mathematical model based on graph theory}
The problem of dissecting rectangles into squares was considered by Brooks, Smith, Stone, and Tutte in \cite{Brooks1987}. Although this 
problem is quite different from the Mondrian problem, its modeling in terms of graphs can be conveniently adapted here. For every squared grid of side length $n$ completely filled with non-overlapping rectangles $R_i$ of dimensions $a_i \times b_i$, for $i \in \{1,\dots,k\}$, we associate a graph $G=(V,A,\omega)$ in the following way: the vertices of $G$ are the horizontal line segments on the squared grid, consisting of a set of horizontal sides of the rectangles $R_i$, sides that  we denote by $P_0,P_1,\dots,P_n$, where $P_0$ and $P_n$ are the upper and lower segments of the squared grid. Every rectangle $R_i$ lies between two horizontal lines. Hence, the arcs of $G$ are precisely the rectangles of the squared grid, that is, $A=\{R_1,R_2,\dots,R_k\}$. We also define the weight of an arc as the dimensions of the given rectangle $\omega(R_i)=(a_i,b_i)$ (see Figure \ref{fig:examp1} for an example).

\begin{figure}[htp]
\centering
\includegraphics[width=16cm]{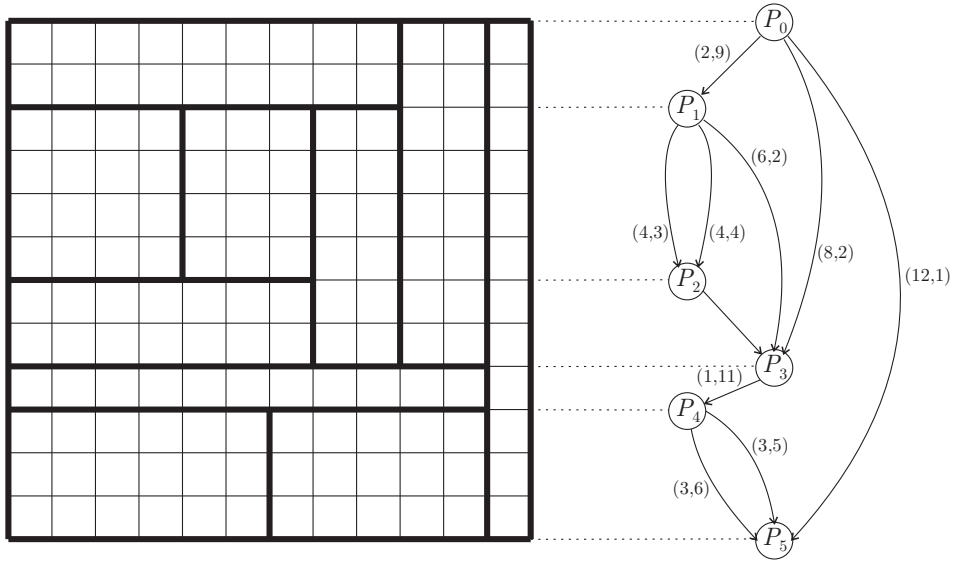}
\vskip-5cm
\caption{A rectangle partition of a square of side length $12$ and its corresponding graph.}\label{fig:examp1}
\end{figure}

Any graph obtained from a rectangle partition of a given square has nice properties. This type of graph is called an {\em electrical network} because its properties can be described in electrical terms (Kirchoff's law, Ohm's law, etc.). Since this mathematical model of the problem is reversible, that is, every rectangle partition of a square can be obtained from an  electrical network, we have a different approach to the problem that could be of some help to make progress. Nevertheless, we do not follow this line of research, but we think it is worth mentioning it here. \\

We organize this paper as follows. In the following section, we provide exact values of the defect $d_k(n)$ for any partition with a small number of given rectangles $k$. We show that optimal Mondrian partitions for squares of side length $3,5$ and $9$ are obtained using exactly $3$ rectangles. A general Mondrian partition of the square is provided in Section \ref{sec:gen}. This construction also shows that the defect ratio $d(n)/n^2$ tends to zero when $n \rightarrow \infty$. Finally, section \ref{sec:perfect} deals with perfect Mondrian partitions, that is, those partitions with defect zero. It is shown that the density of those extremal partitions tends to zero and, in fact, they do not exist when $n$ is the product of at most four different prime numbers.

\section{Bounds on $d(n)$ and optimal partitions}

Giving a rectangle partition of a square of side $n$, we say that this partition is {\em optimal} if it has minimum defect $d(n)$ among all possible rectangle partitions of the square. In particular, if $d(n)=0$ for some partition, then such partition is called a {\em perfect Mondrian partition}. So far, the existence of perfect Mondrian partitions is unknown and the problem of finding optimal partitions seems to be very hard in general. At first sight, the computation of $d(n)$ would require to make all possible partitions, which is devastating from a computational point of view. Besides, finding a direct calculation of $d(n)$ (without such partitions) is a challenging question. \\

Another approach consists of finding upper bounds on $d(n)$. In this section, we present an upper bound based on the number of rectangles of the partition. These upper bounds will be of some help to find optimal partitions for specific values of $n$. For any $k \geq 2$, we define $d_k(n)$ as the minimum defect for all possible partitions of the square of side $n$ using exactly $k$ (non-congruent) rectangles. From its own definition, $d(n) \leq d_k(n)$, and another way of calculating $d(n)$ is
$$
d(n)=\min_k \{d_k(n)\}.
$$
The remaining content of this section is devoted to giving the exact values of $d_2(n)$ and $d_3(n)$. Later on, we will see that $d_3(n)$ gives the exact value of $d(n)$ for $n=3,5$, and $9$. \\

The value $d_2(n)$ is easy to compute. There is (up to a ninety-degree rotation) a unique partition of the square using two rectangles (see Figure \ref{fig:square23} (a)). Let $R_i$ have dimensions $a_i \times b_i$ for $i=1,2$. In this situation, $b_1=b_2=n$ and since $R_1$ is not congruent to $R_2$ ($R_1 \not \cong R_2$), we can choose $a_1 > a_2$. Then,  $d_2(n)=(a_1-a_2)n$. Minimizing $d_2(n)$ is equivalent to find $a_1$ and $a_2$ such that $a_1-a_2$ is minimum, and since they are positive integers satisfying $a_1+a_2=n$, the solution is $(a_1,a_2)=(\frac{n}{2}+1,\frac{n}{2}-1)$ whenever $n$ is even, and $(a_1,a_2)=(\frac{n+1}{2}, \frac{n-1}{2})$ otherwise. That is,
\begin{equation}\label{eq:d2n}
d_2(n)=
\left\{
\begin{array}{ccc}
2n & \textrm{ if } & n \equiv 0 \pmod{2}, \\
n & \textrm{ if } & n \equiv 1 \pmod{2}.
\end{array}
\right.
\end{equation}
\begin{figure}[htb]
\centering
\begin{tabular}{ccc}
\includegraphics[width=.3\textwidth]{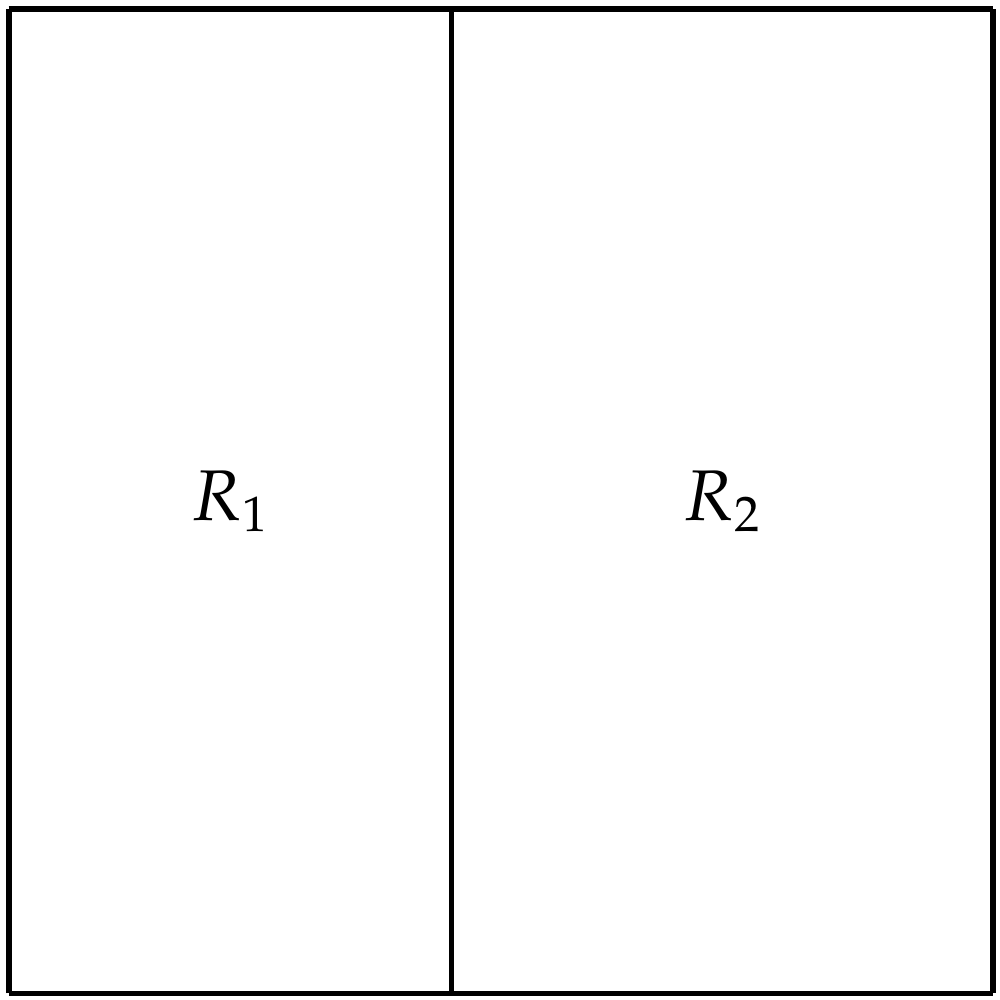} & \includegraphics[width=.3\textwidth]{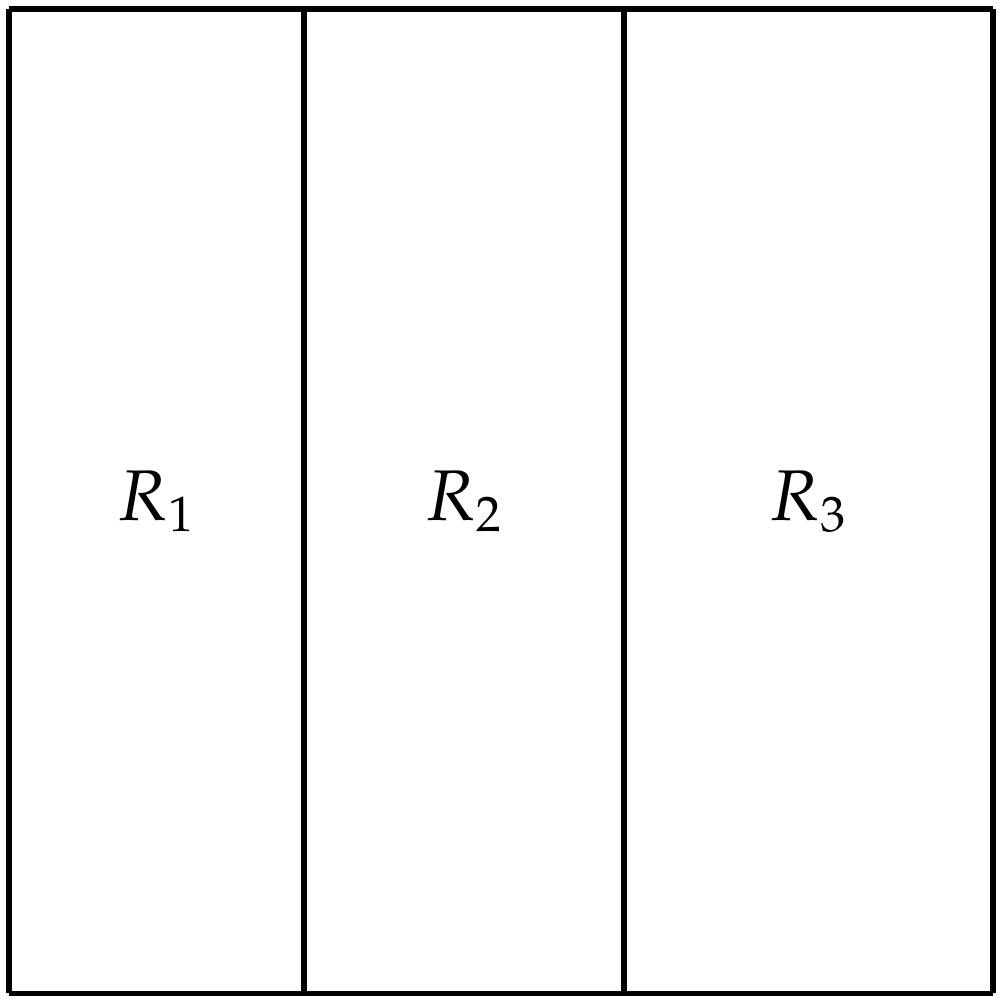} & \includegraphics[width=.3\textwidth]{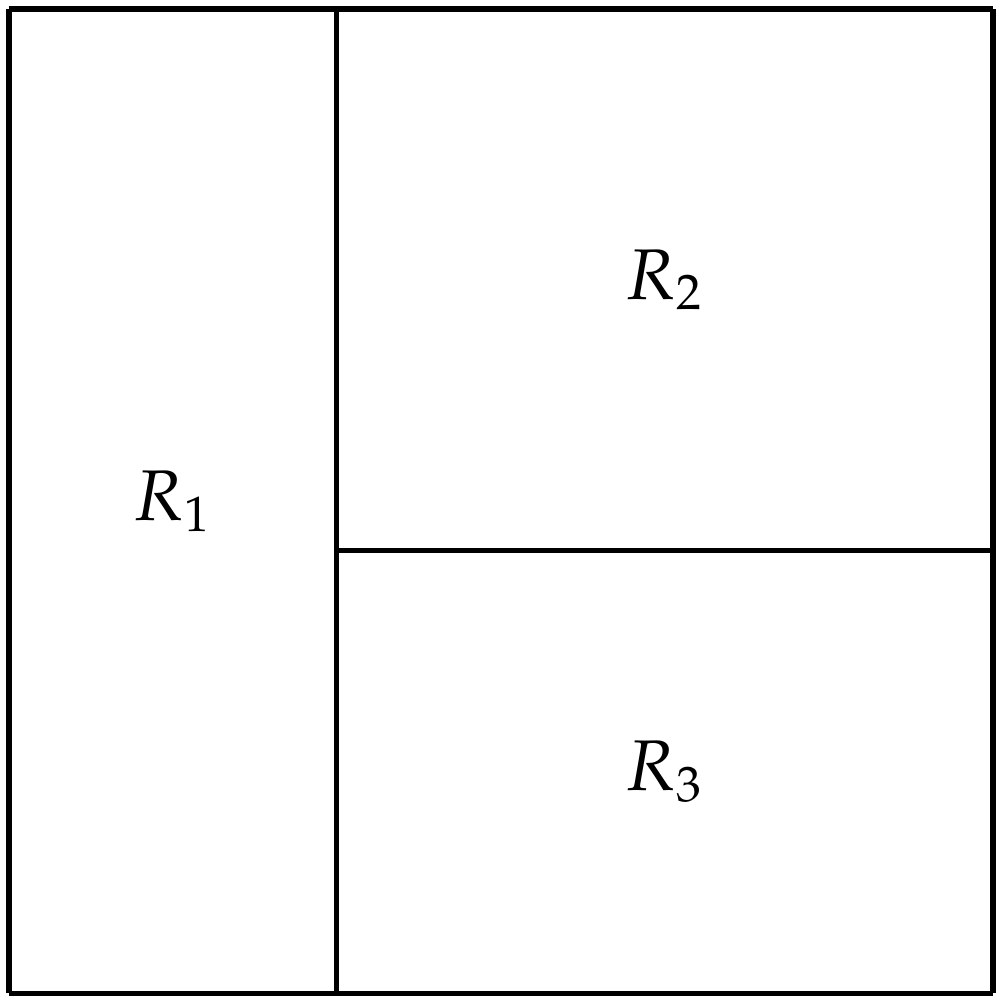} \\
$(a)$ & $(b)$ & $(c)$
\end{tabular}
\vskip.25cm
\caption{Different square partitions using non-congruent rectangles.
}\label{fig:square23}
\end{figure}
The calculation of $d_3(n)$ needs a deeper analysis. We perform it in the following result.
\begin{proposition} Given any positive integer $n \geq 3$, 
\begin{equation}\label{eq:d3n}
 d_3(n)=
\left\{
\begin{array}{ccc}
n+\frac{1}{3}n & \mbox{\emph{ if }} & n \equiv 0 \pmod{6}, \\[0.1cm]
n-\frac{1}{6}(n-1) & \mbox{\emph{ if }} & n \equiv 1 \pmod{6}, \\[0.1cm]
n+\frac{1}{3}(n-2) & \mbox{\emph{ if }} & n \equiv 2 \pmod{6}, \\[0.1cm]
n-\frac{1}{3}n & \mbox{\emph{ if }} & n \equiv 3 \pmod{6}, \\[0.1cm]
n+\frac{1}{3}(n+2) & \mbox{\emph{ if }} & n \equiv 4 \pmod{6}, \\[0.1cm]
n-\frac{1}{6}(n+1) & \mbox{\emph{ if }}& n \equiv 5 \pmod{6}.
\end{array}
\right.
\end{equation}
\end{proposition}
\begin{proof}
There are two types of square partitions using three rectangles (up to rotations and/or symmetries). They are depicted in Figure \ref{fig:square23} $(b)$ and $(c)$. The partition depicted in case $(b)$ can be analyzed using a similar argument to the one given above for two rectangles. For this type of partition, the defect is $3n$ when $n \equiv 0 \pmod{3}$. Otherwise, the defect is even larger. However, we focus on partitions of  type $(c)$ since it gives a lower defect than a partition of type $(b)$ for any $n$, and as a consequence, it produces the values of $d_3(n)$. Let $(a_i,b_i)$ be the dimensions of $R_i$, for $i=1,2,3$. Then, $b_1=b_2+b_3=n$ and $a_1+a_2=a_1+a_3=n$. Since $R_2 \not \cong R_3$, without loss of generality, we can choose $b_2 > b_3$. In this case, we proceed as we did with the case $k=2$. The dimensions $b_2$ and $b_3$ to get the minimum value of $A(R_2)-A(R_3)$ (the difference of the areas of $R_2$ and $R_3$) are given by
\begin{equation}\label{b2b3}
(b_2,b_3)=
\left\{
\begin{array}{ccc}
(\frac{n}{2}+1,\frac{n}{2}-1) & \textrm{ if } & n \equiv 0 \pmod{2}, \\
(\frac{n+1}{2}, \frac{n-1}{2}) & \textrm{ if } & n \equiv 1 \pmod{2}.
\end{array}
\right.
\end{equation}
Notice that the area of $R_1$ must interlace $A(R_2)$ and $A(R_3)$ to get $d_3(n)$ since, otherwise, the defect of the square would be larger. Moreover, since
\begin{equation}\label{ar2ar3}
A(R_2)-A(R_3)=
 \left\{
\begin{array}{ccc}
2(n-a_1) & \textrm{ if } & n \equiv 0 \pmod{2}, \\
n-a_1 & \textrm{ if } & n \equiv 1 \pmod{2},
\end{array}
\right.
\end{equation}
we look for the largest $a_1$ satisfying $A(R_3) \leq A(R_1) \leq A(R_2)$. We divide this problem depending on the parity of $n$:
\begin{itemize}
 \item [(a)] If $n \equiv 0 \pmod{2}$: From $A(R_1) \leq A(R_2)$, $A(R_1) \geq A(R_3)$ and Equation \eqref{b2b3}, we obtain
 \begin{equation}\label{boundsa1}
 \left\{
 \begin{array}{c}
  a_1(\frac{3n}{2}+1) \leq n(\frac{n}{2}+1), \\[.1cm]
  a_1(\frac{3n}{2}-1) \geq n(\frac{n}{2}-1).
\end{array}
\right.
  \end{equation}
To get the largest integer value of $a_1$ satisfying Equation \eqref{boundsa1}, we write $n=6l+p$, where $p \in \{0,2,4\}$. Replacing it in \eqref{boundsa1}, we obtain
\[
\left\{
 \begin{array}{c}
a_1(9l+\frac{3p}{2}+1) \leq 18l^2+(6p+6)l+p(\frac{p}{2}+1), \\[.1cm]
a_1(9l+\frac{3p}{2}-1) \geq 18l^2+(6p-6)l+p(\frac{p}{2}-1).
\end{array}
\right.
\]
It is easy to see that the unique integer value satisfying both inequalities is $a_1=2l$ for $p=0$, and $a_1=2l+1$ for $p=2$ or $p=4$. However, since $A(R_2)-A(R_3)=2(n-a_1)$ (see Equation \eqref{ar2ar3}), we compute $d_3(n)$ using the unique value for $a_1$ obtained as a solution of \eqref{boundsa1}. Hence, for $p=0$, we get $d_3(n)=2(6l-2l)=6l+2l=n+\frac{1}{3}n$. Besides, $d_3(n)=2(2l+2-(2l+1))=n+\frac{1}{3}(n-2)$ for $p=2$ and, for the last case $p=4$, we obtain $d_3(n)=n+\frac{1}{3}(n+2)$.
\item [(b)] If $n \equiv 1 \pmod{2}$: From $A(R_1) \leq A(R_2)$, $A(R_1) \geq A(R_3)$ and Equation \eqref{b2b3}, we get
\begin{equation}\label{2boundsa1}
 \left\{
 \begin{array}{c}
  a_1(3n+1) \leq n(n+1), \\
a_1(3n-1) \geq n(n-1).
\end{array}
\right.
  \end{equation}
Again, we write $n$ as $6l+p$, where $p \in \{1,3,5\}$. Replacing it in \eqref{2boundsa1}, we get
\[
\left\{
 \begin{array}{c}
a_1(18l+3p+1) \leq 36l^2+(12p+6)l+p(p+1), \\
a_1(18l+3p-1) \geq 36l^2+(12p-6)l+p(p-1).
\end{array}
\right.
\]
As in the previous case, the unique value for $a_1$ satisfying Equation \eqref{2boundsa1} is $a_1=2l+1$ for $p=3$. Taking into account that $d_3(n)=A(R_3)-A(R_1)=(n-a_1)$ for this case, we obtain the corresponding value for $d_3(n)$ given in Equation \eqref{eq:d3n}. The cases $p=1$ and $p=5$ have no solution for $a_1$ satisfying both inequalities and they need a different analysis.
\begin{itemize}
 \item Case $p=1$: From $a_1(18l+4) \leq 36l^2+18l+2$, we get $a_1 \leq 2l$. But $a_1=2l$ does not satisfy $a_1(18l+2) \geq 36l^2+6l$. This means that $A(R_1)$ cannot interlace $A(R_2)$ and $A(R_3)$. In fact, $a_1=2l$ is the largest integer value satisfying $A(R_1) \leq A(R_3) \leq A(R_2)$. In this case, the defect is $A(R_2)-A(R_1)=5l+1$. The other possibility is taking $a_1=2l+1$, where  $A(R_3) \leq A(R_2) \leq A(R_1)$, but in this case the defect is $A(R_1)-A(R_3)=8l+1$. As a consequence, $a_1=2l$ and $d_3(n)=n-\frac{1}{6}(n-1)$.
 \item Case $p=5$: From $a_1(18l+16) \leq 36l^2+66l+30$, we get $a_1 \leq 2l+1$. But $a_1=2l+1$ does not satisfy $a_1(18l+14) \geq 36l^2+54l+10$. Moreover, $a_1=2l+1$ is the largest value such that $A(R_1) \leq A(R_3) \leq A(R_2)$ and, for this case, we get a defect of $8l+7$. Nevertheless, taking $a_1=2l+2$, we get $A(R_3) \leq A(R_2) \leq A(R_1)$, and the defect becomes $5l+4$. Hence, $a_1=2l+2$, and $d_3(n)$ becomes $n-\frac{1}{6}(n+1)$.
\end{itemize}

\end{itemize}

\end{proof}

We would like to remark that the previous proof is constructive, that is, it provides, for any $n$, the optimal partitions of a square of side length $n$ using $3$ rectangles. Figure \ref{fig:optsquare} shows these optimal partitions for odd $n$.

\begin{figure}[htb]
\centering
\includegraphics[width=\textwidth]{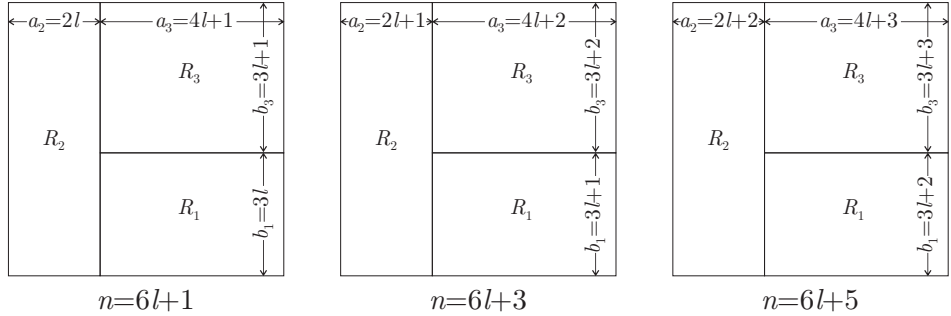}
\vskip-5.5cm
\caption{Optimal partitions of a square of side $n$ using $3$ rectangles, for odd $n$.}\label{fig:optsquare}
\end{figure}

A simple comparison between Equations \eqref{eq:d2n} and \eqref{eq:d3n} shows that $d_3(n) < d_2(n)$, for any $n$ (see also Figure \ref{fig:plotdn}), that is, any optimal partition using three rectangles provides a smaller defect than any optimal partition with two squares. This is something that one can expect. But, on the other side, this monotonicity property does not continue beyond this point. Optimal partitions using just three rectangles appear for some values of $n$, namely $n=3$, $5$,  and $9$. We will prove this result with the help of the next lemma.

\begin{figure}[htp]
\centering
\includegraphics[scale=0.8]{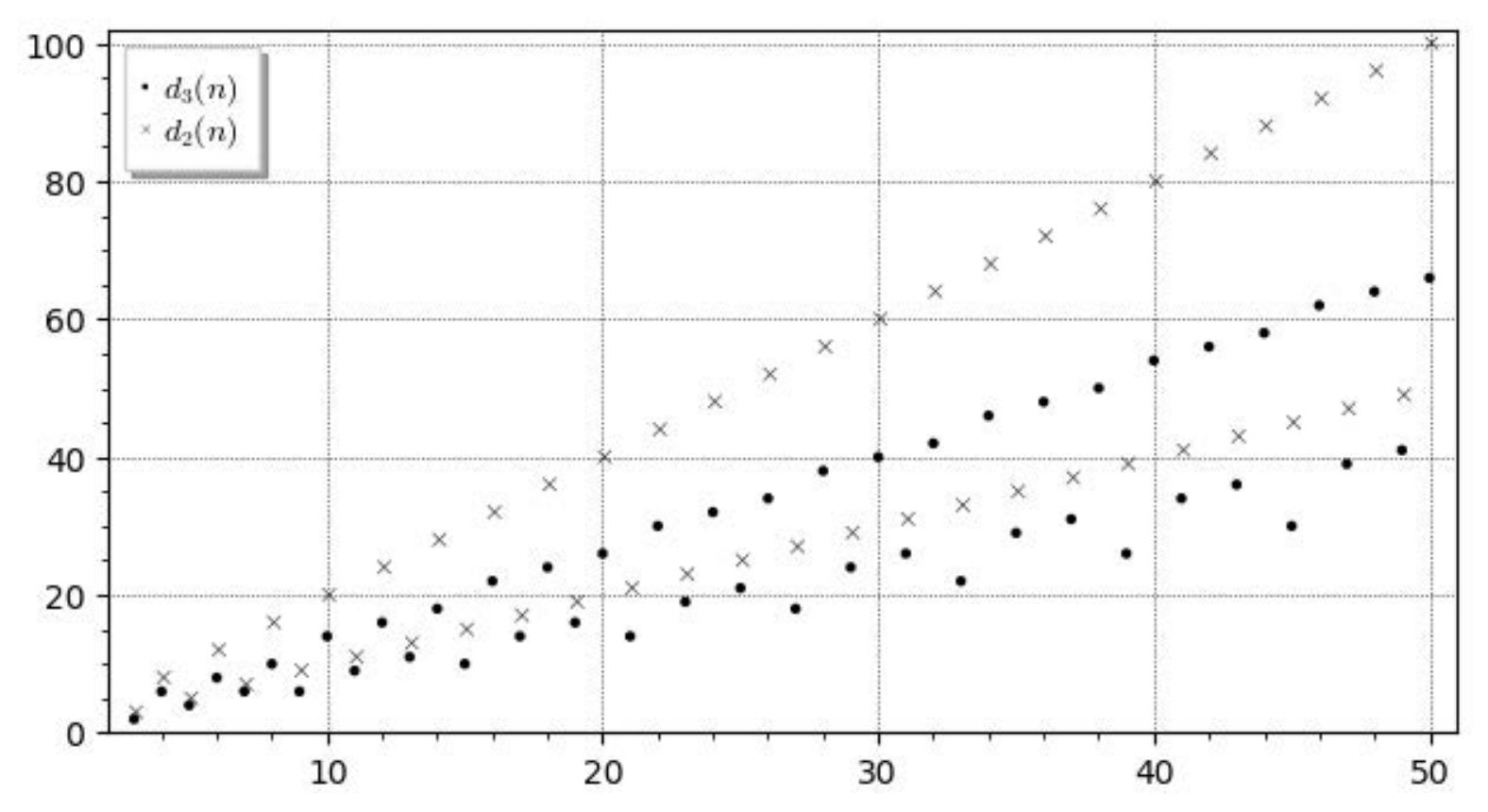}
\caption{Values of $d_2(n)$ and $d_3(n)$ for $n \leq 50$.}\label{fig:plotdn}
\end{figure}

\begin{lemma}\label{lem:bounds}
 Let us consider a partition of a square of side length $n$, for $n \geq 3$, with defect $\delta$ using $k$ rectangles. Let $m$ be the minimum area of such $k$ rectangles. Then, it holds that
 \begin{equation}\label{eq:boundm}
 \left\lceil \frac{n^2-(k-1)\delta}{k} \right\rceil \leq m \leq \left\lfloor \frac{n^2-\delta}{k} \right\rfloor.
  \end{equation}
\end{lemma}
\begin{proof}
Let $A(R_i)$ be the area of the rectangles of the partition, where $1 \leq i \leq k$.
There is at least one rectangle with area $m$, let us say $R_1$, and since the defect is $\delta$, there is at least one other rectangle, let us say $R_k$, with area $m+\delta$. The remaining rectangles $R_i$, for $1 <i<k$, have area $m \leq A(R_i) \leq m+\delta$. This means that
\[
(k-1)m+(m+\delta) \leq \sum_{i=1}^k A(R_i) \leq m + (k-1)(m+\delta).
\]
Taking into account that $\sum_{i=1}^k A(R_i)=n^2$, we get $\delta \leq n^2-km \leq (k-1)\delta$. A simple rearrangement of the previous formula gives the desired result.
\end{proof}

\begin{proposition}
There are optimal partitions of a square of side length $n$ using exactly three rectangles for $n=3,5$ and $9$, that is, $d(n)=d_3(n)$ for $n=3,5$, and $9$.
\end{proposition}
\begin{proof}
Let us suppose that we have a partition of a square of side length $n$ using $k$ rectangles, where $k\geq 4$, and such that $d_k(n)< d_3(n)$, for $n=3,5$, and $9$. We will derive a contradiction. Moreover, for $n=3$ and $5$, we will prove that even $d_k(n) \leq d_3(n)$ is also impossible, showing that optimal partitions for $n=3$ and $5$ only happen with $3$ rectangles.
\begin{itemize}
 \item Case $n=3$. This case is trivial since there is no partition using $k \geq 4$ non-congruent rectangles. 

 \item
 Case $n=5$. Let $m$ be the minimum area of a partition using $k$ rectangles (a $k$-partition for short), with $k \geq 4$. By Lemma \ref{lem:bounds}, we have that $m \leq \lfloor \frac{25-d_k(5)}{k} \rfloor \leq \lfloor \frac{25}{4} \rfloor = 6$. This means that a $k$-partition contains, as a tile of minimum area, one of these tiles: $1 \times 1$, $1 \times 2$, $1 \times 3$, $1 \times 4$, $2 \times 2$, $1 \times 5$, or $2 \times 3$. We derive a contradiction in the case $m \leq 3$. Indeed, the maximum area of a tile would be at most $m+d_k(5)\leq m+d_3(5)=m+4$ (see Equation \eqref{eq:d3n}). If we considered all the available tiles with areas between these values, we would not sum up to $5^2=25$. If $m=4$, the available tiles for this $k$-partition are: $1 \times 4$, $2 \times 2$, $1 \times 5$, $2 \times 3$, and $2 \times 4$ (since the maximum allowed area is $8$). No $k(\geq 4)$ combination of these tiles sums up to $25$. If $m=5$, the tile of dimensions $1 \times 5$ belongs to the $k$-partition. But since the maximum area of a tile would be $9$, in this case, the available tiles would be $1 \times 5$, $2 \times 3$, $2 \times 4$, and $3 \times 3$. Since $k \geq 4$, all of them would perform the partition, but their corresponding areas sum up to $28$. Finally, if $m=6$, there are just four available tiles with dimensions $2 \times 3$, $2 \times 4$, $3 \times 3$, and $2 \times 5$. Their corresponding areas sum up to more than $25$.

\item
Case $n=9$: Assume that there is a $k$-partition ($k \geq 4$) with defect $d_k(9) < d_3(9)=6$. By Lemma \ref{lem:bounds}, we have that a tile of minimum area $m$ satisfies $m \leq \lfloor \frac{81}{4} \rfloor \leq 20$. There are no tiles of areas $m=11,13,17,19$ fitting in the $9 \times 9$-square. For any $m < 12$ or $m=18$, it is easy to show that the sum of the areas of all available tiles (those with areas between $m$ and $m+5$) do not reach $9^2$. The remaining values of $m$ need a specific argument:
\begin{itemize}
\item
$m=12$: The available tiles are $2 \times 6$, $3 \times 4$, $2 \times 7, 3 \times 5, 2 \times 8$, and $4 \times 4$. Their areas sum up to $85$, and the minimum area of a tile is $12$. So, there is no possibility that $k$ of them, for $k \geq 4$, give a partition of a $9 \times 9$-square.
\item
$m=14$: The available tiles are $2 \times 7, 3 \times 5, 2 \times 8, 4 \times 4, 2 \times 9$, and $3 \times 6$. Their areas sum up to $97$. But notice that the tile $2 \times 9$ needs a tile of dimensions $1 \times s$, for some $s \geq 1$, to complete a perfect partition. Since there is no tile of dimensions $1 \times s$, we must remove the tile $2 \times 9$ from the set of available tiles. But, then, the areas of the available tiles sum up to less than $81$.
\item
$m=15$ or $m=16$: The unique tiles of dimension $2 \times s$, for some $s \geq 1$, such that $2s \geq m$ are the tiles of dimensions $2 \times 8$ and $2 \times 9$. The tile $2 \times 9$ cannot perform a perfect partition since, again,  there is no tile of dimension $1 \times s'$. So, the tile $2 \times 8$ does not belong to any perfect partition (it needs another tile of dimension $2 \times s''$, and there is no other tile with such dimensions). The areas of the remaining available tiles do not reach $81$ either for $m=15$ or $m=16$.
\item
$m=20$: The available tiles, in this case, are $4 \times 5$, $3 \times 7$, $4 \times 6$, $3 \times 8$, and $5 \times 5$. Their areas sum up to $117$, and no combination of four of them gives a total area of $81$.
\end{itemize}
\end{itemize}
\end{proof}

\section{A general construction}\label{sec:gen}
Instead of considering the `absolute defect' $d(n)$, which is a difference between areas, it makes sense also to study the `relative defect' or {\em defect ratio}, defined as $\rho(n)=d(n)/n^2$. Here we show that, for $n$ large enough, there exist partitions of the square of side length $n$ with $\rho(n)<\epsilon$ for any $\epsilon>0$.

\begin{proposition} For any natural number $r \geq 3$, there exists a partition of the square of side length $n=\frac{k!}{2}$, where $k=2r+1$, using $k$ rectangles such that the defect ratio $\rho(n)\rightarrow 0$ when $r\rightarrow \infty$.
\end{proposition}

 \begin{proof}
Given $n=\frac{k!}{2}$, let us consider the rectangles $R_1,R_2,R_3,\ldots,R_k$ inside the square following a `spiral' pattern starting from $R_1$ in the basis of the square and ending in the center of the square (see Figure \ref{integ-decomp(k=7)} as an example for $k=7$). The idea of this proof is as follows: $R_1$ is obtained as a $1/k$ part of the square, $R_2$ as a $1/(k-1)$ part of the rest, $R_3$ as a $1/(k-2)$ part of the rest, and so on.
The dimensions $a_i \times b_i$ of the rectangle $R_i$ are given in Table \ref{tab:spiral},  for all $i=1,2,\dots,k$.
\begin{table}[htb]\label{tab:spiral}
\[
{
\extrarowheight = -0.5ex
\renewcommand{\arraystretch}{2.7}
\begin{array}{|c|c|c|}
\hline\hline
\textrm{Rectangle} & 1\textrm{st side length} & 2\textrm{nd side length} \\
\hline\hline
R_1 & a_1=n & b_1=\dfrac{n^2}{ka_1}=\dfrac{n}{k}\\
\hline
R_2 & a_2=\dfrac{n}{k-1} & b_2=\dfrac{n^2}{ka_2}=\dfrac{n(k-1)}{k}\\
\hline
R_3 & b_3=\dfrac{b_2}{k-2}=\dfrac{n(k-1)}{k(k-2)} & a_3=\dfrac{n^2}{kb_3}=\dfrac{k-2}{n(k-1)} \\
\hline
R_4 & \displaystyle{a_4=\frac{n}{k-3}} & b_4=\dfrac{n^2}{ka_4}=\dfrac{n(k-3)}{k}\\
\hline
R_5 & b_5=\dfrac{b_4}{k-4}=\dfrac{n(k-3)}{k(k-4)} & a_5=\dfrac{n^2}{kb_5}=\dfrac{n(k-4)}{k-3}\\
\hline
\vdots & \vdots & \vdots \\
\hline
R_{k-3} & a_{k-3}=\dfrac{n}{n-2} & b_{k-3}=\dfrac{n^2}{ka_{k-3}}\\
\hline
R_{k-2} & b_{k-2}=\dfrac{n}{n-1} & a_{k-2}=\dfrac{n^2}{kb_{k-2}}\\
\hline
R_{k-1} & a_{k-1}=\dfrac{a_{k-2}-1}{2} & \displaystyle{b_{k-1}=n-\sum_{i=1}^{r}b_{2i-1}} \\
\hline
R_k & a_{k}=\dfrac{a_{k-2}+1}{2}  & \displaystyle{b_{k}=b_{k-1}} \\
\hline\hline
\end{array}
}
\]
\caption{Dimensions of the sequence of rectangles $R_1,\dots,R_k$ following a spiral pattern.}
\end{table}

First, notice that every rectangle dimension is indeed a natural number. Moreover, the first $k-2$ rectangles have the same area, that is,
\[
A(R_1)=\dots=A(R_{k-2})=\dfrac{1}{4}k\left[(k-1)!\right]^2,
\]
and the last two rectangles have area
\[
\begin{array}{ll}
A(R_{k-1})=\dfrac{1}{4}\left( k\left[(k-1)!\right]^2+\left[(k-1)!!\right]^2 \right) & \textrm{and} \\
A(R_{k})=\dfrac{1}{4}\left( k\left[(k-1)!\right]^2-\left[(k-1)!!\right]^2 \right), &
\end{array}
\]
where $(k-1)!!=(k-1)(k-3)(k-5)\cdots 2$ (note that $k$ is odd). Hence, for this partition we have defect  $A(R_{k-1})-A(R_{k})=\frac{[(k-1)!!]^2}{2}$.
As a consequence $d_k(n) = \frac{[(k-1)!!]^2}{2}$ for $n=\frac{k!}{2}$. Now,

$$\rho(n) = \frac{d(n)}{n^2} \leq \frac{d_k(n)}{n^2} \leq \left(\frac{2(k-1)!!}{k!} \right)^2= \left( \frac{2r!!}{(2r+1)!} \right)^2 \rightarrow 0 \ \textrm{ when } \ r\rightarrow \infty.
$$

\end{proof}

\begin{figure}[h!]
\centering
\includegraphics[scale=0.6]{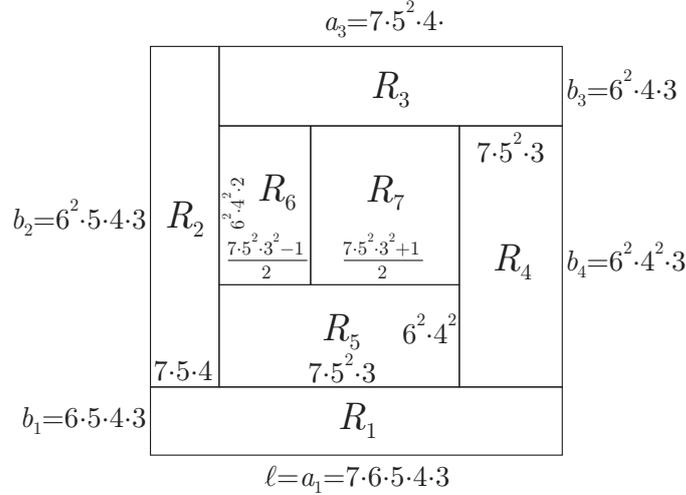}
\vskip-10cm
\caption{A `balanced' Mondrian  integer partition of a square in $k=7$ rectangles. The areas are not proportional to their values, as they should be, to make clearer the drawing.}
\label{integ-decomp(k=7)}
\end{figure}


\begin{table}[h!]\label{t:so-integer}
\centering
\begin{tabular}{|l|l|}
\hline\hline
$a_1 =2520$ & $b_1 =360 $\\
\hline
$a_2 = 420$ & $b_2 = 2160$\\
\hline
$a_3 = 2100 $ & $b_3 = 432$\\
\hline
$a_4 =525$ & $b_4 = 1728$\\
\hline
$a_5 = 1575$ & $b_5 =  576$\\
\hline
$a_6 = 787 $&$ b_6 = 1152 $\\
\hline
$a_7 =788 $ & $b_7 = 1152$\\
\hline\hline
\end{tabular}
\caption{Integral dimensions of the $k=7$ rectangles in a `balanced'  Mondrian partition.}
\end{table}

\section{On perfect Mondrian partitions}\label{sec:perfect}
Here we investigate the possible existence of perfect Mondrian partitions. In this context, O'Kuhn \cite{okuhn2018mondrian} proved that, given $x>3$, the number of side lengths $n(x)\le x$  of squares not having such a partition, that is, $n(x)=|\{n\le x:d(n)\neq 0\}|$,  satisfies a lower bound of the order
\begin{equation}
\label{O'Kuhn}
n(x)\ge C\frac{x \log(\log(x))}{\log(x)},
\end{equation}
where $C\approx\frac{1}{e^{2\gamma}\log(2)}$, and $\gamma=0.577215665\ldots$ is the constant of Euler-Mascheroni.\\
\\
We prove another similar bound by giving  some necessary conditions for the existence of the Mondrian partitions with zero defect.
Although our techniques are very simple, the obtained bound overcomes O'Kuhn's one, as we will see later.
\begin{theorem}
\label{theo-n(x)}
Let ${\cal N}(s_1,\ldots,s_m)$ be the set of integers
of the form $n=p_1^{s_1}\cdots p_m^{s_m}$, where $p_1<\cdots <p_m$ are some primes with fixed powers $s_1,\ldots,s_m$.
\begin{itemize}
\item[$(i)$]
For any given $s_1,\ldots, s_m$, there is a constant $c$ such that, if $p_1>c$, then no square $S$ with side length $n\in {\cal N}(s_1,\ldots,s_m)$ has a perfect Mondrian partition.
\item[$(ii)$]
There is no perfect Mondrian partition of a square $S$ with side length $n$ a primer power $(n\in {\cal N}(s_1))$.
\item[$(iii)$]
There is no perfect Mondrian partition of a square $S$ with side length $n$ being  a product of 2, 3, or 4 different primes $(n\in {\cal N}(1,1)\cup {\cal N}(1,1,1))\cup {\cal N}(1,1,1,1))$.
\item[$(iv)$]
There is no perfect Mondrian partition of a square $S$ with side length $n$ being  a product of 2 or 3  primes.
\end{itemize}
\end{theorem}
\begin{proof}
We start by proving $(ii)$ since it is a special case.
Let $S$ be a square of side length $n=p^s$, with $p$ a prime, having a perfect Mondrian partition ${\cal P}$ with $k$ (non-congruent) rectangles $R_1,\ldots, R_k$. Then, $A(S)=p^{2s}$ and, since the partition is perfect, $k$ must divide $n$. So, we can write $k=p^r$ for some integer $r\le 2s$. Moreover, the area of each rectangle is $A(R_i)=\frac{n^2}{k}=p^{2s-r}$ for $i=1,\ldots,k$.
Now, since the partition is Mondrian, all the sides $x_i$, $y_i$ of the rectangles must be different, and of the form $x_i=p^{\sigma_i}$ and $y_i=p^{\tau_i}$ with $\sigma_i+\tau_i=2s-r$ for $i=1,\ldots,k$. Furthermore, such sides cannot be larger than the side $n=p^s$ of $S$. Thus, the total number of `available' squares is the number
$$
K:=|\{(x_i,y_i): \{x_i, y_i\}\cap\{x_i, y_i\}=\emptyset, \mbox{for $i \neq j$}, x_i+y_i=2s-r, x_i,y_i\le s\}|.
$$
These pairs are: $(s,s-r)$, $(s-1,s-r+1)$, \ldots, $(\frac{2s-r}{2},\frac{2s-r}{2})$ if $r$ is even;
and $(s,s-r)$, $(s-1,s-r+1)$, \ldots, $(\frac{2s-r+1}{2},\frac{2s-r-1}{2})$ if $r$ is odd.
So, $K=\frac{r+2}{2}$ if $r$ is even; and $K=\frac{r+1}{2}$ if $r$ is even. Then, from $k\le K$,
we get the inequalities
$$
p^r\le \frac{r+2}{2}\quad\mbox{if $r$ is even,}\qquad \mbox{and}\qquad
p^r\le \frac{r+1}{2}\quad \mbox{if $r$ is odd.}
$$
But such inequalities are not satisfied for any $p\ge 2$ and $r\ge 2$, so that there is no perfect Mondrian partition of a square with side a prime power.\\
\\
Now, let us prove the more general case $(i)$.
Let $S$ be a square of side length $n=p_1^{s_1}\cdots p_m^{s_m}$, with primes $p_1<\cdots< p_m$, and let $s=\max\{s_1,\ldots,s_m\}$. As before, assume that $S$ has a perfect Mondrian partition ${\cal P}=\{R_1,\ldots, R_k\}$. Now, $A(S)=p_1^{2s_1}\cdots p_m^{2s_m}$ and, as the partition is perfect, $k$ must divide $n$. Then, let $k=p_1^{r_1}\cdots p_m^{r_m}$ for some integers $r_i\le 2s_i$, for $i=1,\ldots,m$, and the area of each rectangle is constant, namely $A(R_j)=\frac{n^2}{k}=p_1^{2s_1-r_1}\cdots p_m^{2s_m-r_m}$ for $j=1,\ldots,k$.
Again, all the sides $x_j$, $y_j$ of the rectangles must be different, with $x_jy_j=A(R_j)$, and $x_j,y_j\le n$. Let us see that the total number $K$ of `available' squares satisfies the bound
\begin{equation}
\label{bound-K}
K\le \frac{1}{2}\prod_{i=1}^m (2s_i-r_i+1)-\prod_{i=1}^m (s_i-r_i+1) +1,
\end{equation}
if $\prod_{i=1}^m (s_i-r_i+1) >1$, and
\begin{equation}
\label{bound-Kb}
K\le \frac{1}{2}\prod_{i=1}^m (2s_i-r_i+1),
\end{equation} 
otherwise.
Indeed, the first term in \eqref{bound-K}--\eqref{bound-Kb} stands for {\bf all} the possible pairs $(x_j,y_j)$
such that $x_j=p_1^{2s_1-r_1}\cdots p_m^{2s_m-r_m}$ with $r_i\in \{0,\ldots, 2s_i-r_i\}$, for $i=1,\ldots,m$ (note that the total number of such $r_i$ is $2s_i-r_i+1$). But,  since each $x_j$ must not be larger than
$n=p_1^{s_1}\cdots p_m^{s_m}$, the second term in \eqref{bound-K}, if it is at least $2$, it means that we can subtract all the values $x_j=p_1^{t_1}\cdots p_m^{t_m}$ with $t_i\in\{s_i,\ldots, 2s_i-r_i\}$ for $i=1,\ldots,m$ (the number of such $t_i$ is $s_i-r_i+1$) in all cases except for one. This exception is for the case $t_i=r_i$ for $i=1,\ldots,m$ (that is,
$x_j=p_1^{s_1}\cdots p_m^{s_m}=n$), which justifies the last term $+1$ in \eqref{bound-K}.\\
\\
Now, on the one hand, we have that the number $k(>1)$ of rectangles
satisfies $k\ge |{\cal N}(1,0,\ldots,0)|= p_1$ (since $p_1<\ldots<p_m$),
and, on the other hand, the maximum number of available rectangles, with $r_1=1$ and $r_2=\cdots=r_m=0$, satisfies
\begin{equation}
\label{bound2-K}
K\le K'=\frac{1}{2}(2s+1)^{m-1}2s-(s+1)^{m-1}s +1,
\end{equation}
(since $s\ge s_i$ for $i=1,\ldots,m$ ).
Consequently, if $k\ge p_1> K'\ge K$, we have a contradiction and $S$ cannot have a perfect Mondrian partition.\\
\\
In particular, notice that,  by taking $s_1=\cdots=s_m=1$, the bound \eqref{bound2-K} becomes
\begin{equation}
\label{bound3-K}
K\le 3^{m-1}-2^{m-1}+1,
\end{equation}
 but, as we will see next, a more precise analysis can be done.
\\
With this aim, note that, in this case, $n=p_1p_2\cdots p_m$. If the number of rectangles is the minimum $k=p_1$, then the area of each rectangle is $A(R_j)=x_jy_j=p_1p_2^2\cdots p_m^2$ for $j=1,\ldots,m$, with $x_j,y_j\le p_1p_2\cdots p_m$. Now, all the possible pairs $(x_j,y_j)$ are $2\cdot 3^{m-1}/2=3^{m-1}$, as expected from \eqref{bound3-K}. To this quantity, we can subtract all the values $x_j=p_1^{t_1}\cdots p_m^{t_m}$ with $t_1\in\{0,1\}$, $t_j\in\{0,1,2\}$ for $j=2,\ldots,m$ such that $t_1+\cdots+t_m\ge m$ and with the condition that, if $t_i=0$ for some $i$, there is at least one $j>i$ such that $j=2$ (excepting  the case $t_1=\cdots =t_m=1$). Alternatively, the second condition is equivalent to require that the partial sums $t_1+\cdots +t_i$, for $i=1,\ldots,m$, reach their maximum at $i=m$.   Note that, since $p_1<p_2<\cdots<p_m$, each of such $n$-tuples $(t_1,\ldots,t_m)$ satisfies $p_1^{t_1}\cdots p_m^{t_m}>n$, as required.
At this point, it is useful to characterize each possible $m$-tuple as  an $m$-path in a grid with integer points, starting from $(0,0)$, and using $i$-th step
$D=(1,-1)$ if $t_i=0$, $H=(1,1)$ if $t_i=1$, and $U_i=(1,1)$ if $t_i=2$, for $i=2\ldots,m$; and only $D$ or $H$ for the first step.
(U=up-step of slope $1$, D=down-step of slope $-1$, H=horizontal-step of slope $0$.)
Then, if, for the moment, we include the case $HH...H$ ($t_1=\cdots=t_m=1)$, the numbers $a(m)$ of such paths are
\begin{equation}
\label{a(m)}
a(1)=1,\  a(2)=3,\  a(3)=8,\  a(4)=22,\  a(5)=61,\  a(6)=171, \ldots
\end{equation}
The good news is that this is a known sequence of the numbers, 
corresponding to the so-called {\em Moztkin paths}. In the
 {\em On-line Encyclopedia of Integer Sequences} \cite{Sl}, this is the sequence A025566, where the following formula is provided (starting from $m=1$ instead of $m=0$, as it is presented in  \cite{Sl}):
\begin{equation}
\label{f:a(m)}
a(m)=\sum_{k=0}^{\lfloor\frac{m+1}{2}\rfloor} {{m-1}\choose{k}} {{m-k+1}\choose{k}}.
\end{equation}
To prove that the number of our $m$-tuples $(t_1,\ldots,t_m)$ coincides with $a(m)$ is better to look at a related sequence that also appears in \cite{Sl},  labeled with A005773. Namely, for $m=1,2,\ldots,b(m)$, it corresponds to the number of paths contained in an $m\times m$ grid from $(0,0)$ to the line $x=m-1$, using also the steps $U$, $H$, and $D$, as above. (In this case, only $H$ or $U$ are allowed as first step.) The sequence is now
\begin{equation}
\label{b(m)}
b(1)=1,\  b(2)=2,\  b(3)=5,\  b(4)=13,\  b(5)=35,\  b(6)=96,\ b(7)=267, \ldots
\end{equation}
and it corresponds to the formula
\begin{equation}
\label{f:b(m)}
b(m)=\sum_{k=0}^{\lfloor\frac{m}{2}\rfloor} {{m-1}\choose{k}} {{m-k}\choose{k}}.
\end{equation}
For example, the paths for $m=3,4$ are shown in Figure \ref{fig:paths}.
\begin{figure}[t]
\vskip .5cm
\centering
\includegraphics[width=12cm]{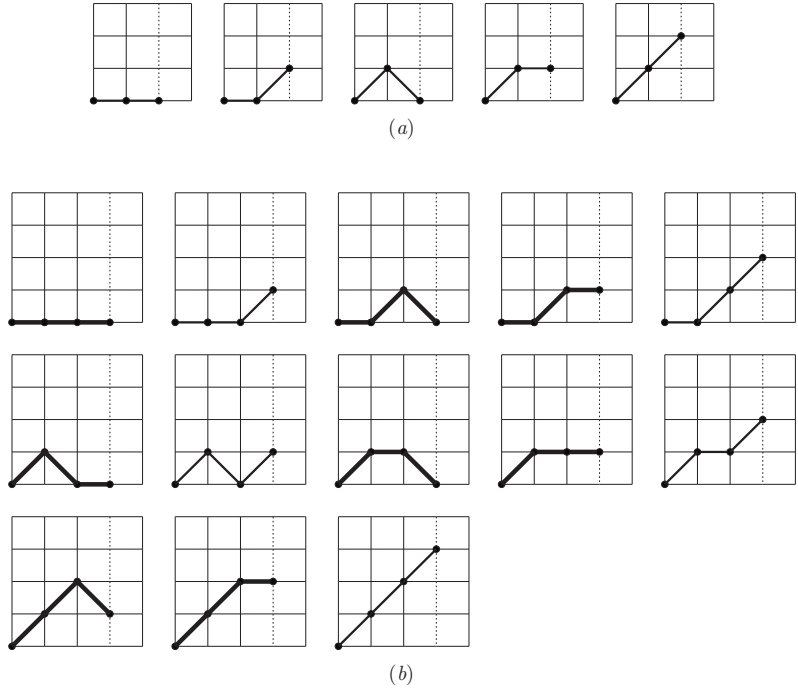}
\caption{Paths from $(0,0)$ to the line $x=m-1$ for $(a)$ $m=3$, and $(b)$ $m=4$. }\label{fig:paths}
\end{figure}
Comparing the sequences \eqref{a(m)} and \eqref{b(m)}, the reader can note that $a(m)=b(m+1)-b(m)$. The reason for this is clear if, for instance, we look at Figure \ref{fig:paths}. Our paths representing the desired $m$-tuples, say $(t_1,t_2,t_3)$ $(m=4)$, are those of Figure \ref{fig:paths} $(b)$, not ending with the step $U$, and looked `upside down' (those drown in bold line). Besides, the other paths, after removing their last step $U$ are in correspondence with those of   \ref{fig:paths} $(a)$.
Summarizing,  we proved that $a(m)=b(m+1)-b(m)$ is the number we are looking for (of course, this equality can be checked by using \eqref{f:a(m)} and \eqref{f:b(m)}.)
Thus, for any square $S$ with side $n=p_1\cdots p_m$, where $p_1<\cdots< p_m$, we conclude that, if
\begin{equation}
\label{cond-on-p1}
p_1>3^{m-1}-\sum_{k=0}^{\lfloor\frac{m+1}{2}\rfloor} {m-1\choose{k}} {{m-k+1}\choose{k}}+1,
\end{equation}
where the last one stands for the available case $x_i=p_1\cdots p_m$,
then $S$ has not a perfect Mondrian dissection with $k\ge p_1$ rectangles.

In particular, $(iii)$ is proved by taking $m=2,3,4$ in \eqref{cond-on-p1}, which gives the conditions
$k\ge p_1>1$, $k\ge p_1>2$, and $k\ge p_1>6$, respectively. Besides, in Dalf\'o, Fiol, L\'opez, and Mart\'{\i}nez-P\'erez \cite{dflm20}, it was proved that there are no perfect Mondrian dissections of a square with $k\le 8$ rectangles.

Finally, let us prove $(iv)$. The case of $s=2$  primes follows immediately from $(ii)$ and $(iii)$. For the case of $s=3$,
 the same lines of reasoning allow us to prove the following.
\begin{itemize}
\item
If $k\ge p_1>8$, there is no perfect Mondrian partition of a square $S$ with side length $n=p_1^2p_2$,  that is,  $n\in {\cal N}(2,1)$.
\item
If $k\ge p_1>6$, there is no perfect Mondrian partition of a square $S$ with side length $n=p_1p_2^2$,  that is,  $n\in {\cal N}(1,2)$.
\end{itemize}
Indeed, for instance, in the first case, $A(S)=p_1^4 p_2^2$ and, if $k=p_1$, then $A(R_i)=p_1^3 p_2^2$. Thus, the number of  possible side lengths is $4\cdot 3- 4=8$ (since the four sides $p_1p_2^2$, $p_1^2p_2^2$, $p_1^3p_2$, and $p_1^3p_2^2$ are not allowed because they are greater than $n$).
Thus, $(iv)$ follows again from the referred results in  \cite{dflm20}.
\end{proof}

Let us see a  consequence of our approach.
Following the conventional notation, let $\pi_s(n)$ be the number of integers not greater than $n$, with exactly $s$ prime divisors (not necessarily distinct). Such numbers are usually called {\em $s$-almost primes}. For instance, every number in ${\cal N}(s_1,\ldots,s_m)$ is an {\em $s$-almost primes} with $s=s_1+\cdots+s_m$. A classic result of Landau states that    $\pi_s(n)$ is asymptotic to 
\begin{equation}
\label{landau}
 \pi_s(n) \sim \frac{n}{\log n}\frac{(\log(\log(n)))^{s-1}}{(s-1)!},
\end{equation}
see, for instance, Tenenbaum \cite{t15}.
\begin{corollary}
Let  $n(x)$ be the number of side lengths at most $x$  of squares not having such a partition, as before. 
Then, its `density' $n(x)/x$ is asymptotic to
$$
\frac{n(x)}{x}\sim \frac{(\log(\log(x)))^2}{2\log(x)}.
$$
\end{corollary}
\begin{proof}
By Theorem \ref{theo-n(x)} $(iv)$, we know  that there is no perfect Mondrian partition of a square $S$ of side $n$ being the product of $s(\le 3)$ primes (not necessarily distinct).  Then, by using \eqref{landau}, the number $n(x)$ of such squares, with side $n\le x$,  is asymptotic to
\begin{equation}
\label{n(x)}
n(x)\sim \sum_{s=1}^{3} \frac{x}{\log x}\frac{(\log(\log(x)))^{s-1}}{(s-1)!},
\end{equation}
and the result follows by considering the term of the greatest order (with $s=3$).
\end{proof}
For the sake of comparison, notice that the result of O'Kuhn in \eqref{O'Kuhn} corresponds to take only the first two terms (for $s=1,2$) of the sum in \eqref{n(x)}.


\subsection*{Acknowledgments}
The research of C. Dalf\'o and N. L\'opez has been partially supported by grant MTM2017-86767-R (Spanish Ministerio de Ciencia e Innovaci\'on). The research of C. Dalf\'o and M. A. Fiol has been partially supported by
AGAUR from the Catalan Government under project 2017SGR1087 and by MICINN from the Spanish Government under project PGC2018-095471-B-I00.


\end{document}